\documentclass[a4paper,12pt]{article}
\usepackage{array,amsfonts,amsmath,graphicx,mathrsfs,amssymb,amsthm,latexsym}
\usepackage[latin1]{inputenc}
\usepackage{color}

\usepackage{csquotes}

\newtheorem{thm}{Theorem}[section]

\newtheorem{lemma}[thm]{Lemma}

\def \cH {{\cal H}}

\begin{document}
\title{Uniformly resolvable decompositions of $K_v$ into $1$-factors and odd $n$-star factors }

\author
 {Jehyun Lee and  Melissa Keranen\\
\small Department of Mathematical Sciences \\
\small Michigan Technological University\\
}

\maketitle

\vspace{5 mm}

\begin{abstract}
We consider uniformly resolvable decompositions of $K_v$ into subgraphs such that each resolution class contains only blocks isomorphic to the same graph. We give a partial solution for the case in which all resolution classes are either $K_2$ or $K_{1,n}$ where $n$ is odd.

\end{abstract}

\section{Introduction}\label{intro}

Let $G=(V,E)$ be a graph. An $\cH$-$decomposition$ of the graph $G$ is a collection of edge disjoint subgraphs $\cH = \{H_1,H_2, \dots, H_a \}$ such that every edge of $G$ appears in exactly one graph $H_i \in \cH$. 

The subgraphs, $H_i \in \cH$, are called blocks. An $\cH$-decomposition is called resolvable if the blocks in $\cH$ can be partitioned into classes (or factors) $F_i$, such that each $F_i$ is a spanning subgraph of $G$. A resolvable $\cH$-decomposition is also called an $\cH$-factorization of $G$, and its classes are referred to as $\cH$-factors. An $\cH$-decomposition is called $uniformly$ $resolvable$ if each factor $F_i$ consists of blocks that are all isomorphic.

In the case of uniformly resolvable $\mathcal{H}$-decompositions, existence results can be categorized based on $\mathcal{H}$: when $\cH$ is a set of two complete graphs of order at most five in \cite{DLD, R, SG, WG}; when $\cH$ is a set of two or three paths on two, three or four vertices in \cite{GM1,GM2, LMT}; for $\cH =\{P_3, K_3+e\}$ in \cite{GM}; for $\cH =\{K_3, K_{1,3}\}$ in \cite{KMT}; for $\cH =\{C_4, P_{3}\}$ in \cite{M}; and for $\cH =\{K_3, P_{3}\}$ in \cite{MT}.

If $\cH=\{H_1,H_2 \}$, then we may also consider how many factors contain copies of $H_1$ and how many factors contain copies of $H_2$. We let $(H_1,H_2)$-$URD(v;r,s)$ denote a uniformly resolvable decomposition of $K_v$ into $r$ classes containing only copies of $H_1$ and $s$ classes containing only copies of $H_2$. In this paper, we consider the existence problem for $\{H_1,H_2\}=\{K_2,K_{1,n}\}$.

While the general case $(K_2,K_{1,n})$-$URD(v;r,s)$ is still open and in progression, we have observed that the standard methods used for most cases of $(r,s)$ are not applicable to solve the cases when number of $1$-factors is small. Thus, we studied these cases separately. With regards to the extremal cases, we have the following. 

A $K_2$-factorization of $G$ is known as a $1$-{\em factorization} and its factors are called 1-{\em factors}. We let $I$ denote a $1$-factor. It is well known that a 1-factorization of $K_v$  exists if and only if $v$ is even (\cite{Lu}).

If $n=3$, that is, the case of $(K_2,K_{1,3})$-$URD(v;r,s)$, necessary and sufficient conditions for the existence of the decomposition was given in \cite{CC}. If $n=4$, that is, the case of $(K_2,K_{1,4})$-$URD(v;r,s)$, necessary and sufficient conditions for the existence of the decomposition was given in \cite{MD}. If $n=5$, that is, the case of $(K_2,K_{1,5})$-$URD(v;1,s)$, necessary and sufficient conditions for the existence of the decomposition was given in \cite{JAY}. A generalization of this to the case of $(K_2,K_{1,n})$-$URD(v;1,s)$ for odd $n \geq 3$ is also completely solved in \cite{JAY2}.

In this paper, we focus on the $(K_2,K_{1,n})$-$URD(v;r,s)$ for all $(r,s)$ pairs where $r,s \geq 1$, and we give a partial solution to the existence problem of a $(K_2,K_{1,n})$-URD$(v;$ $r,s)$.\\

\section{Necessary Conditions}

\begin{lemma} 
\label{ness}
Let $n \geq 3$ be an odd integer. If a $(K_2,K_{1,n})-URD(v;r,s)$ exists, then there is an integer $x$, $0\leq x \leq \lfloor \frac{v-1}{2n} \rfloor$, such that $s=(n+1)x$ and $r=v-1-2nx$. Further, $v \equiv 0 \pmod{2}$ if $r>0$ and $v\equiv 0 \pmod{(n+1)}$ if $s>0$. 
\end{lemma}

\begin{proof}
Assume that there exists a $(K_2,K_{1,n})-URD(v;r,s).$ By counting the number of edges of $K_v$ that appear in the factors it follows that

\begin{align*}
r\frac{v}{2}+s\frac{nv}{n+1}=\frac{v(v-1)}{2},
\end{align*}
and hence
\begin{align}
(n+1)r+2ns=(n+1)(v-1).\label{eq1}
\end{align}
Let $S$ be the set of $s$ $K_{1,n}$-factors, and let $R$ be the set of $r$ 1-factors. Because the factors of $R$ are regular of degree 1, every vertex of $K_v$ is incident to $r$ edges in $R$ and $(v-1)-r$ edges in $S$. Assume that any fixed vertex appears in $x$ factors of $S$ with degree $n$ and in $y$ factors of $S$ with degree 1. Because

\begin{align*}
x+y=s \textrm{ and } nx+y=v-1-r,    
\end{align*}
equation (\ref{eq1}) gives

\begin{align*}
(n+1)(v-1-nx-y)+2n(x+y) = (n+1)(v-1).
\end{align*}

This implies $y=nx$ and $s=(n+1)x$.

Further, replacing $s=(n+1)x$ in Equation (1) provides $r=v-1-2nx$, where $x\leq \frac{v-1}{2n}$ (because $r$ is a non-negative integer).

Finally, if $r > 0$, then $v$ must be even; while if $s>0$, then necessarily $n+1$ divides $v$ (because $K_{1,n}$ is a graph on $n+1$ vertices).

\end{proof}

If there exists a $(K_2,K_{1,n})-URD(v;r,s)$ with $s=0$, the result is a $1$-factorization(see \cite{Lu}). So, we will consider the cases with $s > 0$. Therefore, $v \equiv 0 \pmod{n+1}$, and we will prove the existence of a $(K_2,K_{1,n})$-$URD(v;r,s)$ for all possible $(r,s) \in \mathcal{J}=\{(r,s)|r=v-1-2nx, s=(n+1)x, \textrm{ with } 0\leq x \leq \lfloor \frac{v-1}{2n} \rfloor\}$.

\section{Weighted Graphs and Preliminary Results}

Let $G$ be a graph, and $t$ be a positive integer. A $\textit{weighted graph}$ $G_{(t)}$ is a graph on $V(G) \times \mathbb{Z}_t$ with edge set $\{\{x_i,y_j\}:\{x,y\} \in E(G), i,j \in \mathbb{Z}_t \}$. We refer to the construction of $G_{(t)}$ from $G$ as \enquote{giving weight $t$ to $G$}.

For some positive integer $m$, let $K_m$ be a complete graph. Then, for some positive integer $n$, let $K_{m(n+1)}$ be the graph obtained by giving weight $n+1$ to $K_m$. Then for each $x \in V(K_m)$, let $K_{n+1}^x$ denote a complete graph with vertex set $V(K_{n+1}^x) = \{x_i | x \in K_m, i\in \mathbb{Z}_{n+1} \}$. Note that each $K_{n+1}^x$ are mutually disjoint. Thus, for $v=m(n+1)$, we can view the complete graph $K_v$ as $K_v = K_{m(n+1)} \cup \Big(\bigcup\limits_{x\in V(K_m)}K_{n+1}^x \Big)$.

For our purposes, we will decompose $K_{m(n+1)}$ into weighted cycles $C_{m(n+1)}$. We begin with two well-known results about the decomposition of complete graphs into cycles.

\begin{lemma} 
\label{oddC}
(Alspach, Brian (2001)) 
: For positive odd integers $m$ and $n$ with $3\leq m \leq n$, the graph $K_{n}$ can be decomposed into cycles of length $m$ if and only if the number of edges in $K_{n}$ is a multiple of $m$.
\end{lemma}

\begin{lemma} 
\label{evenC}
(Alspach, Brian (2001)) 
: For positive even integers $m$ and $n$ with $4\leq m \leq n$, the graph $K_{n}-I$ can be decomposed into cycles of length $m$ if and only if the number of edges in $K_{n}-I$ is a multiple of $m$.
\end{lemma}

It is obvious that $m$ divides $|E(K_m)|$, thus if $m$ is odd, by Lemma~\ref{oddC} we can decompose $K_m$ into $m$-cycles. The number of $m$-cycles in the decomposition is:
\begin{align*}
|E(K_m)|=\frac{m(m-1)}{2}\cdot \frac{1}{m}=\frac{m-1}{2}.
\end{align*}

Similarly if $m$ is even, by Lemma~\ref{evenC}, we can decompose $K_m$ into $I$ and $\frac{m-2}{2}$ $m$-cycles.

Now give weight $n+1$ to $K_m$ and to each $m$-cycle, $C_m$ to obtain $K_{m(n+1)}$ and copies of weighted $m$-cycle, $C_{m(n+1)}$. An $m$-cycle and a weighted $m$-cycle is illustrated in Figure~\ref{v27}. Hence, we have a decomposition of $K_{m(n+1)}$ into $\frac{m-1}{2}$ weighted $m$-cycles, $C_{m(n+1)}$ for odd $m \geq 3$. We also have a decomposition of $K_{m(n+1)}$ into $I_{(n+1)}$ and $\frac{m-2}{2}$ weighted $m$-cycles, $C_{m(n+1)}$, for even $m \geq 4$. Let $C_{m(n+1)}^k$ denote the $k^{th}$ weighted $m$-cycle. If $m\geq3$ is odd, then $i \in \{1,2,\dots,\frac{m-1}{2} \}$ and if $m\geq4$ is even, then $k \in \{1,2,\dots,\frac{m-2}{2} \}$. Note that all $C_{m(n+1)}^i$ have the same vertex set, but they are all mutually edge disjoint subgraphs of $K_{m(n+1)}$.

By this decomposition, we now view $K_v = (K_{m(n+1)}) \cup \Big(\bigcup\limits_{x\in V(K_m)}K_{n+1}^x\Big)$ as follows. 
\begin{align*}
    &K_v = (K_{m(n+1)}) \cup \Big(\bigcup\limits_{x\in V(K_m)}K_{n+1}^x\Big) \\
        &=
        \begin{cases}
           \Big( \cup_{k=1}^{\frac{m-1}{2}} C_{m(n+1)}^k \Big) \cup \Big(\bigcup\limits_{x\in V(K_m)}K_{n+1}^x\Big),  &\textrm{ if $m$ is odd }\\
           \Big( \cup_{k=1}^{\frac{m-2}{2}} C_{m(n+1)}^k \Big) \cup \Big(\bigcup\limits_{x\in V(K_m)}K_{n+1}^x\Big) \cup (I_{(n+1)}),  &\textrm{ if $m$ is even }
    \end{cases}
\tag{2}
\end{align*}

\begin{figure}[!htb]
    \centering
    \includegraphics[width=\textwidth]{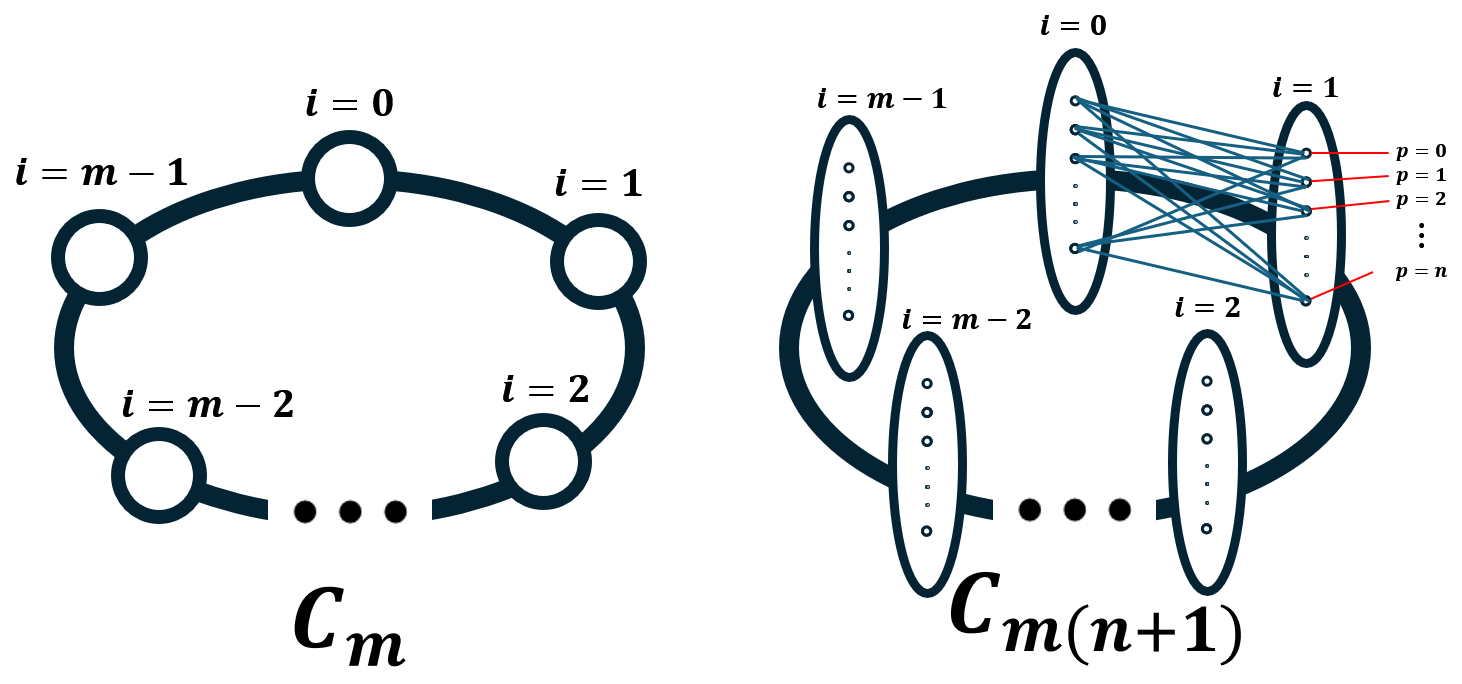}
    \caption{A cycle $C_m$ and a weighted cycle $C_{m(n+1)}$}
    \label{v27}
\end{figure}

\subsection{Almost Uniformly Resolvable Decompositions}
For a weighted graph $G_{(t)}$ on the vertex set $V(G) \times \mathbb{Z}_t$, let $J(G_{(t)})$ be a subgraph graph of $G_{(t)}$ with $V(J(G_{(t)}))=V(G_{(t)})$ and edge set $\{\{x_i,y_i\}:\{x,y\} \in E(G), i \in \mathbb{Z}_{t} \}$. Let $H = G_{(t)}-J(G_{(t)})$ be the graph on vertex set $V(G) \times \mathbb{Z}_t$ with edge set $\{\{x_i,y_j\}:\{x,y\} \in E(G), i,j \in \mathbb{Z}_{t}, i \not = j \}$. If an $(X,Y)-URD(H;r,s)$ exists, then we say that $G_{(t)}$ has an \textit{almost uniformly resolvable decomposition}, denoted by $(X,Y)$-$AURD(G_{(t)};r,s)$.

In this section, we construct {\em AURD} of given weighted cycles and of a weighted edge.

Suppose $C_m=(0,1,\dots,m-1)$ is an $m$-cycle and $C_{m(n+1)}$ is the corresponding weighted $m$-cycle(with weight $n+1$). Then edges in $C_m$ are directed edges, for example, $(x,x+1)$ with $x\in \{0,1,2,\dots,m-1\}$. The vertex set of $C_{m(n+1)}$ is
\[ V(C_{m(n+1)}) = V(C_m) \times \mathbb{Z}_{n+1} \]
and the edge set of $C_{m(n+1)}$ is
\[ E(C_{m(n+1)} = \{ (x_i,(x+1)_j) | (x,x+1)\in E(C_m); i,j \in \mathbb{Z}_{n+1} \}. \]
We define the difference of an edge $(x_i,(x+1)_j)$ to be $d = j-i \pmod{n+1}$. 

Our first result gives the equivalent of the decomposition of $C_{m(n+1)}$ into $1$-factors. However, for our purposes, it is vital that we view it as an almost uniformly resolvable decomposition.

\begin{lemma}
\label{oneF}
A $(K_2,K_{1,n})-AURD(C_{m(n+1)};2n,0)$ exists for any odd integer $n$ and integer $m\geq3$. 
\end{lemma}

\begin{proof}

If $m \geq 3$ is even, we will construct a $(K_2,K_{1,n})-AURD(C_{m(n+1)};2n,0)$ by the following method. Without loss of generality, assume $C_m=(0,1,\dots,m-1)$ is the $m$-cycle. Let $C_{m(n+1)}$ be the weighted $m$-cycle(with weight $n+1$).\\

If $d \in \{1,2,\dots, n\}$ is odd, we will construct a pair of $1$-factors $B_{1_{a},d}$ and $B_{1_{b},d}$ of $C_{m(n+1)}$ as follows. For any $d$, let
\begin{align*}
B_{1_a,d} =& \{ ((x)_i,(x+1)_{i+d}) | x \in \mathbb{Z}_m \textrm{ and even } i \in \mathbb{Z}_{n+1} \} \\
B_{1_b,d} =& \{ ((x)_i,(x+1)_{i+d}) | x \in \mathbb{Z}_m \textrm{ and odd }  i \in \mathbb{Z}_{n+1} \} .
\end{align*}

If $d \in \{1,2,\dots,n\}$ is even, we construct a pair of $1$-factors $B_{2_{a},d}$ and $B_{2_{b},d}$ of $C_{m(n+1)}$ as follows. For any $d$, let
\begin{align*}
B_{2_a,d} =& \{ ((x)_i,(x+1)_{i+d}), ((x+1)_{i+1},(x+2)_{i+d+1}) | \textrm{ even } x \in \mathbb{Z}_m \textrm{ and even } i \in \mathbb{Z}_{n+1} \} \\
B_{2_b,d} =& \{ ((x)_i,(x+1)_{i+d}), ((x+1)_{i+1},(x+2)_{i+d+1}) | \textrm{ even } x \in \mathbb{Z}_m  \textrm{ and odd } i \in \mathbb{Z}_{n+1}  \} \\
\end{align*}

By this construction, for a given $d \in \{1,2,\dots,n\}$, we obtain two $1$-factors of $C_{m(n+1)}$ which contains all the edges with difference $d$. Because we can construct two $1$-factors for each $d$, we have successfully constructed a $(K_2,K_{1,n})-AURD(C_{m(n+1)};2n,0)$. \\

If $m \geq 3$ is odd, we will construct a $(K_2,K_{1,n})-AURD(C_{m(n+1)};2n,0)$ by the following method. \\

If $n+1 \equiv 2 \pmod{4}$, let $D = \{1,2,\dots,n\}$, and let $D'=\{d|d \equiv 3 \pmod{4} \textrm{ and } d\in D\}$. For any $d \in D'$, we will match $d$ with two even differences $d-1$ and $d+1$ to construct three pairs of $1$-factors $B_{3_a,d}$ and $B_{3_b,d}$, $B_{4_a,d}$ and $B_{4_b,d}$, and $B_{5_a,d}$ and $B_{5_b,d}$ of $C_{m(n+1)}$ as follows. For each $d \in D'$, let
\begin{align*}
B_{3_a,d} = \{ ((0)_i,(1)_{i+d}), ((x)_i,(x+1)_{i+(d-1)}), ((x+1)_{i+1},(x+2)_{i+1+(d-1)}) |&\\
 \textrm{  odd } x \in \mathbb{Z}_m, \textrm{ even } i \in &\mathbb{Z}_{n+1}  \} \\
B_{3_b,d} = \{ ((0)_i,(1)_{i+d}), ((x)_i,(x+1)_{i+(d-1)}), ((x+1)_{i+1},(x+2)_{i+1+(d-1)}) | &\\
 \textrm{ odd } x \in \mathbb{Z}_m,  \textrm{ and odd } i  \in &\mathbb{Z}_{n+1} \},
\end{align*}
\begin{align*}
B_{4_a,d} = \{ ((1)_i,(2)_{i+d}), ((x)_i,(x+1)_{i+(d+1)}), ((x+1)_{i+1},(x+2)_{i+1+(d+1)}) |&\\
  x \not=0, \textrm{ even } x \in \mathbb{Z}_m , \textrm{ and even } i \in &\mathbb{Z}_{n+1} \}\\
B_{4_b,d} = \{ ((1)_i,(2)_{i+d}), ((x)_i,(x+1)_{i+(d+1)}), ((x+1)_{i+1},(x+2)_{i+1+(d+1)}) | &\\
  x \not=0, \textrm{ even } x \in \mathbb{Z}_m,  \textrm{ and odd } i \in &\mathbb{Z}_{n+1} \},
\end{align*}
\begin{align*}
B_{5_a,d} = \{ ((0)_i,(1)_{i+(d-1)}),((1)_{i+1},(2)_{(i+1)+(d+1)}), ((x)_i,(x+1)_{i+d}) |& \\
  x \not= 0,1, x \in \mathbb{Z}_m, \textrm{ and even } i \in &\mathbb{Z}_{n+1} \}\\
B_{5_b,d} = \{ ((0)_i,(1)_{i+(d-1)}),((1)_{i+1},(2)_{(i+1)+(d+1)}), ((x)_i,(x+1)_{i+d}) | &\\
x \not=0,1, x \in \mathbb{Z}_m,  \textrm{ and odd } i \in &\mathbb{Z}_{n+1} \}.
\end{align*}

Then, for any given odd $d \in D \backslash D'$, that is $d \equiv 1 \pmod{4}$, we construct a pair of $1$-factors $B_{6_{a},d}$ and $B_{6_{b},d}$ of $C_{m(n+1)}$ as follows. For each $d\in D \backslash D'$, let
\begin{align*}
B_{6_a,d} =& \{ ((x)_i,(x+1)_{i+d}) | x \in \mathbb{Z}_m  \textrm{ and odd } i \in \mathbb{Z}_{n+1} \} \\
B_{6_b,d} =& \{ ((x)_i,(x+1)_{i+d}) | x \in \mathbb{Z}_m  \textrm{ and even }  i \in \mathbb{Z}_{n+1} \}. 
\end{align*}

By this construction, for a given $d \in \{1,2,3,\dots,n\}$, we obtain two $1$-factors of $C_{m(n+1)}$ that contain all edges with difference $d$. Thus, we have a $(K_2,K_{1,n})-AURD(C_{m(n+1)};2n,0)$ for any $n+1 \equiv 2\pmod{4}$. \\

If $n+1\equiv 0 \pmod{4}$, let $D= \{1,2,\dots,n\}$. In this case, $D$ contains an odd number of even differences. So, to pair two even differences with one odd difference as the previous construction, we will construct two $1$-factors $B_{7_{a},2}$ and $B_{7_{b},2}$ of $C_{m(n+1)}$ which only contain the edges with the difference $d=2$ only.
\begin{align*}
B_{7_a,2} = \{ ((x)_i,(x+1)_{i+2}), ((x)_{(i+1)},(x+1)_{(i+1)+2})  |&\\
x \in \mathbb{Z}_m, i  \in &\mathbb{Z}_{n+1}, \textrm{ and }  i \equiv 0 \pmod{4} \} \\
B_{7_b,2} = \{ ((x)_i,(x+1)_{i+2}), ((x)_{(i+1)},(x+1)_{(i+1)+2}) | &\\
x \in \mathbb{Z}_m, i  \in &\mathbb{Z}_{n+1}, \textrm{ and }  i \equiv 2 \pmod{4} \}.
\end{align*}

 Then, $D \backslash \{2\}$ contains an even number of even differences. Now, let $D' = \{d|d\equiv 1 \pmod{4}, d\in D, \textrm{ and } d \not= 1 \}$. For any $d \in D'$, we match up with two even differences $d-1$ and $d+1$ to construct three pairs of $1$-factors $B_{8_a,d}$ and $B_{8_b,d}$, $B_{9_a,d}$ and $B_{9_b,d}$, and $B_{10_a,d}$ and $B_{10_b,d}$ of $C_{m(n+1)}$ as follows:
\begin{align*}
B_{8_a,d} = \{ ((0)_i,(1)_{i+d}), ((x)_i,(x+1)_{i+(d-1)}), ((x+1)_{i+1},(x+2)_{i+1+(d-1)}) |&\\
  \textrm{ odd } x \in \mathbb{Z}_m   \textrm{ and even } i \in& \mathbb{Z}_{n+1} \} \\
B_{8_b,d} = \{ ((0)_i,(1)_{i+d}), ((x)_i,(x+1)_{i+(d-1)}), ((x+1)_{i+1},(x+2)_{i+1+(d-1)}) |&\\
  \textrm{ odd } x \in \mathbb{Z}_m   \textrm{ and odd } i \in& \mathbb{Z}_{n+1} \},
\end{align*}
\begin{align*}
B_{9_a,d} = \{ ((1)_i,(2)_{i+d}), ((x)_i,(x+1)_{i+(d+1)}), ((x+1)_{i+1},(x+2)_{i+1+(d+1)}) |&\\
  x\not=0, \textrm{ even } x \in \mathbb{Z}_m,    \textrm{ and even } i \in& \mathbb{Z}_{n+1} \} \\
B_{9_b,d} = \{ ((1)_i,(2)_{i+d}), ((x)_i,(x+1)_{i+(d+1)}), ((x+1)_{i+1},(x+2)_{i+1+(d+1)}) |&\\
  x\not=0, \textrm{ even } x \in \mathbb{Z}_m,   \textrm{ and odd } i \in& \mathbb{Z}_{n+1} \},
\end{align*}
\begin{align*}
B_{10_a,d} = \{ ((0)_i,(1)_{i+(d-1)}),((1)_{i+1},(2)_{(p+1)+(d+1)}), ((x)_i,(x+1)_{i+d}) |&\\
  x\not=0,1, x \in \mathbb{Z}_m,    \textrm{ and even } i \in& \mathbb{Z}_{n+1} \}\\
B_{10_b,d} = \{ ((0)_i,(1)_{i+(d-1)}),((1)_{i+1},(2)_{(p+1)+(d+1)}), ((x)_i,(x+1)_{i+d}) |&\\
  x \not= 0,1, x \in \mathbb{Z}_m,    \textrm{ and odd } i \in& \mathbb{Z}_{n+1} \} .
\end{align*}

Then, for any given odd $d \in D \backslash D'$, we construct a pair of $1$-factors $B_{11_{a},d}$ and $B_{11_{b},d}$ of $C_{m(n+1)}$ as follows:
\begin{align*}
B_{11_a,d} =& \{ ((x)_i,(x+1)_{i+d}) |   x \in \mathbb{Z}_m   \textrm{ and odd } i \in \mathbb{Z}_{n+1}\} \\
B_{11_b,d} =& \{ ((x)_i,(x+1)_{i+d}) |   x \in \mathbb{Z}_m    \textrm{ and even }  i \in \mathbb{Z}_{n+1}\} 
\end{align*}

By this construction, for a given $d \in \{1,2,3,\dots,n\}$, we obtain two $1$-factors of $C_{m(n+1)}$ that contain all edges with difference $d$. Thus, we have a $(K_2,K_{1,n})-AURD(C_{m(n+1)};2n,0)$ for any $n+1\equiv 0\pmod{4}$. \\

Hence, a $(K_2,K_{1,n})-AURD(C_{m(n+1)};2n,0)$ exists for any odd $n>1$.

\end{proof}

\begin{lemma}
\label{starF}
 : A $(K_2,K_{1,n})-AURD(C_{m(n+1)};0,n+1)$ exists for any odd integer $n>1$ and integer $m\geq3$.
\end{lemma}

\begin{proof}
Assume $C_m=(0,1,\dots,m-1)$ is the $m$-cycle and $C_{m(n+1)}$ is the weighted $m$-cycle. We give $n+1$ $n$-star factors as follows. For each $j \in \mathbb{Z}_{n+1}$, let
\begin{align*}
S_{j} =& \{ (x_j;(x+1)_{j+1}, (x+1)_{j+2},\dots,(x+1)_{j+n}) |   x \in \mathbb{Z}_m \}.
\end{align*}

Then, all edges of $E(C_{m(n+1)})$ with difference $d \in \{1,2,3,\dots,n\}$ appear exactly once in some $S_j$. Hence, a $(K_2,K_{1,n})-AURD(C_{m(n+1)};0,n+1)$ exists for any odd $n>1$.

\end{proof}

Let $m\geq4$ be an even integer, and consider the weighted graph $K_{m(n+1)}$ on the vertex set $V(K_m) \times \mathbb{Z}_{n+1}$. Let $I$ be a $1$-factor of $K_m$ with $\frac{m}{2}$ edges $\{x,y\} \in E(I)$ and $I_{(n+1)}$ be the corresponding weighted $1$-factor of $K_{m(n+1)}$. Then, for each $d\in \{1,2,\dots,n\}$, we can take the following $1$-factor of $I_{(n+1)}$:
\begin{align*}
B_d =& \{ \{(x)_i,(y)_{i+d}\}| i \in \mathbb{Z}_{n+1} \}.
\end{align*}

The only edges from $I_{(n+1)}$ that do not appear are edges from $J(I_{(n+1)})=\{ \{x_i,y_i\} | \{x,y\} \in E(I), i\in \mathbb{Z}_{(n+1)} \}$. Therefore, we have the following result.

\begin{lemma}
\label{edgeF}
 A $(K_2,K_{1,n})-AURD(I_{(n+1)};n,0)$ exists for any odd integer $n>1$.
\end{lemma}

By Lemma $~\ref{oneF}$ and Lemma $~\ref{starF}$, we've shown that $C_{m(n+1)}-J(C_{m(n+1)})$ can be decomposed into $2n$ $1$-factors or into $n+1$ $n$-star factors. In order to find uniformly resolvable decompositions of $K_v$, we must now turn into finding decompositions of $J(K_{m(n+1)}) \cup \Big(\bigcup\limits_{x\in V(K_m)}K_{n+1}^x\Big)$.

\section{Difference 0 and Inner Edges}
Recall, by Lemmas~\ref{oddC} and~\ref{evenC}, $K_{m(n+1)}$ has been decomposed as follows:

\begin{align*}
    K_{m(n+1)} =
        \begin{cases}
           \Big( \cup_{k=1}^{\frac{m-1}{2}} C_{m(n+1)}^k \Big),  &\textrm{ if $m$ is odd }\\
           \Big( \cup_{k=1}^{\frac{m-2}{2}} C_{m(n+1)}^k \Big)  \cup (I_{(n+1)}),  &\textrm{ if $m$ is even}.
    \end{cases}
\end{align*}

Denote each weighted cycle by $C_{m(n+1)}^k$ where $1\leq k \leq t$ with $t = \frac{m-1}{2}$ if $m$ is odd and $t = \frac{m-2}{2}$ if $m$ is even. Recall that all $C_{m(n+1)}^k$ have the same vertex set, but they are all mutually edge disjoint subgraphs of $K_{m(n+1)}$. Because each $C_{m(n+1)}^k$ has either a $(K_2,K_{1,n})-AURD(C_{m(n+1)};0,n+1)$ decomposition or a $(K_2,K_{1,n})-AURD(C_{m(n+1)};2n,0)$ decomposition, it follows that $\cup_{k=1}^{t} C_{m(n+1)}^k$ has an almost uniformly resolvable decomposition into $r$ $K_2$-factors and $s$ $K_{1,n}$-factors where $(r,s) \in \mathcal{J}=\{(r,s)|r=2nx, s=(n+1)(t-x),  \textrm{ for any non negative integer } x \leq t \}$.

Recall, an $AURD$ of a weighted graph $K_{m(n+1)}$ exists if a $URD$ of $H$ exists where $H=K_{m(n+1)}-J(K_{m(n+1)})$. In Lemmas~\ref{oddfilling} and~\ref{evenfilling}, we will decompose $\overline{H} = K_v-H$ into $1$-factors. Note that $\overline{H} = J(K_{m(n+1)} \cup \Big(\bigcup\limits_{x\in V(K_m)}K_{n+1}^x\Big)$.

\begin{lemma} 
\label{oddfilling}
Let $v=m(n+1)$ for some odd integer $m\geq 3$, and odd integer $n\geq3$. A $(K_2,K_{1,n})-URD(\overline{H};m+(n-1),0)$ exists.
\end{lemma}

\begin{proof}
For any edge $\{x,y\} \in K_m$, let $R^{x,y} = \{ \{x_{i},y_{i}\} |   i \in \mathbb{Z}_{n+1} \}$ denote a set of $n+1$ edges from $E(J(K_{m(n+1)})$. For any vertex $x \in V(K_m)$, let $R^{x} = \{ \{x_{i},x_{i+1}\} |   \textrm{ even } i \in \mathbb{Z}_{n+1} \}$ denote a set of edges from $E(K_{n+1}^x)$. Then, for each $x \in K_m$, define $A_x \cup B_x$ to be a $1$-factor of $\overline{H}$ as follows:
\begin{align*}
A_{x} =& \{R^{x-1,x+1}, R^{x-2,x+2}, R^{x-3,x+3}, \cdots, R^{x-\frac{m-1}{2},x+\frac{m-1}{2}}\} \\
B_{x} =& \{R^{x}\}
\end{align*}
This produces $m$ $1$-factors of $\overline{H}$. Note that, for any $x\in V(K_m)$, $B_x$ is a $1$-factor of $K_{n+1}^x$. Since $n+1$ is even, it is trivial to decompose $K_{n+1}^x-B_x$ into $n-1$ $1$-factors. Let $B_x^k$ be the $k^{th}$ $1$-factor in this decomposition. Then for $k=1,2,\dots,n-1$, $\cup_{x\in V(K_m)} B_x^k$ gives the remaining $n-1$ $1$-factors of $\overline{H}$. Thus, a $(K_2,K_{1,n})-URD(\overline{H};m+n-1,0)$ exists.


\end{proof}

\begin{lemma} 
\label{evenfilling}
Let $v=m(n+1)$ for some even integer $m\geq 4$ and some odd integer $n\geq3$. A $(K_2,K_{1,n})-URD(\overline{H};m+n-1,0)$ exists.
\end{lemma}

\begin{proof}

Because $m$ is even, there exists a $1$-factorization of $K_m$ with $m-1$ $1$-factors. Let $F_1,F_2,\dots,F_{m-1}$ denote the $m-1$ $1$-factors. For each edge $\{x,y\} \in F_k$, where $k \in \{1,2,\dots,m-1\}$, let $R^{x,y}=  \{ \{x_{i},y_{i}\} |   i \in \mathbb{Z}_{n+1} \}$ be a set of edges from $J(K_{m(n+1)})$. Then, for $k=1,2,\dots,m-1$, define $A_k=\cup_{\{x,y\}\in F_k} R^{x,y}$ to be a $1$-factor of $J(K_{m(n+1)})$. This produces a total of $m-1$ $1$-factors of $\overline{H}$.

For any $x\in V(K_m)$, there exists a $1$-factorization of $K_{n+1}^x$ because $n$ is odd. Let $F_k^x$ denote the $k^{th}$ $1$-factor of the $1$-factorization. Then, for $k=1,2,\dots,n$, define $B_k=\cup_{x \in V(K_m)} F_k^x$ to be a $1$-factor of $\cup_{x \in V(K_m)} K_{n+1}^x$. Thus, we obtain another $n$ $1$-factors of $\overline{H}$.

\end{proof}

\section{Results}
Let $v=m(n+1)$. Recall how we view $K_v$ in Equation $(2)$ in Lemma~\ref{evenC}. We will finalize the decomposition of $K_v$, dealing with the cases for $m$ odd and $m$ even cases separately.

\begin{thm} 
\label{oddResult}
Let $v=m (n+1)$ for any odd integer $n \geq 3 $ and any odd integer $m \geq 3$. A $(K_2,K_{1,n})-URD(K_v;r,s)$ exists for all pairs $(r,s) \in \mathcal{J}=\{(r,s)|r=2n\ell+(m+n-1), s=(n+1)(\frac{m-1}{2}-\ell)$ for any some non-negative integer $\ell \leq \frac{m-1}{2} \}$.
\end{thm}

\begin{proof}
Recall from the discussion following Lemma~\ref{evenC}, that we are viewing $K_v$ as:
\begin{align*}
K_v &= (K_{m(n+1)}) \cup \Big(\bigcup\limits_{x\in V(K_m)}K_{n+1}^x\Big) \\
&= \Bigg(\bigcup\limits_{k=1}^{\frac{m-1}{2}} C_{m(n+1)}^k \Bigg) \cup \Big(\bigcup\limits_{x\in V(K_m)}K_{n+1}^x\Big)
\end{align*}
By Lemmas~\ref{oneF} and Lemma~\ref{starF}, each $C_{m(n+1)}$ can be decomposed into either a $(K_2,K_{1,n})-AURD(C_{m(n+1)};2n,0)$ or a $(K_2,K_{1,n})-AURD(C_{m(n+1)};0,n+1)$. Let $\ell$(with $0 \leq \ell \leq \frac{m-1}{2}$) be the number of weighted $m$-cycles in which we choose to a $(K_2,K_{1,n})-AURD(C_{m(n+1)};2n,0)$. Then, $(\frac{m-1}{2}-\ell)$ weighted cycles will be chosen to have a $(K_2,K_{1,n})-AURD(C_{m(n+1)};0,n+1)$. By this method, we obtain $2n\ell$ $1$-factors and $(n+1)(\frac{m-1}{2}-\ell)$ $n$-star factors. Then, by Lemma~\ref{oddfilling}, we obtain $(m+n-1)$ more $1$-factors, and this completes the decomposition.

\end{proof}

\begin{thm} 
\label{evenResult}
Let $v=m\cdot(n+1)$ for any odd integer $n \geq 3 $ and any even integer $m \geq 4$. A $(K_2,K_{1,n})-URD(K_v;r,s)$ exists for all pairs $(r,s) \in \mathcal{J}=\{(r,s)|r=2n\ell+(m+2n-1), s=(n+1)(\frac{m-2}{2}-\ell)$ for any non-negative integer $0 \leq \ell \leq \frac{m-2}{2}\}$ .
\end{thm}

\begin{proof}
Recall from the discussion following Lemma~\ref{evenC}, that we are viewing $K_v$ as:
\begin{align*}
K_v &= (K_{m(n+1)}) \cup \Big(\bigcup\limits_{x\in V(K_m)}K_{n+1}^x\Big)\\ 
&= \Bigg( \bigcup\limits_{k=1}^{\frac{m-2}{2}} C_{m(n+1)}^k \Bigg) \cup \Big(\bigcup\limits_{x\in V(K_m)}K_{n+1}^x\Big) \cup (I_{(n+1)})
\end{align*}
By Lemmas~\ref{oneF} and Lemma~\ref{starF}, each $C_{m(n+1)}$ can be decomposed into either a $(K_2,K_{1,n})-AURD(C_{m(n+1)};2n,0)$ or a $(K_2,K_{1,n})-AURD(C_{m(n+1)};0,n+1)$. Let $\ell$(with $0 \leq \ell \leq \frac{m-2}{2}$) be the number of weighted $m$-cycles in which we choose a $(K_2,K_{1,n})-AURD(C_{m(n+1)};2n,0)$. Then, $(\frac{m-2}{2}-\ell)$ weighted cycles will be chosen to have a $(K_2,K_{1,n})-AURD(C_{m(n+1)};0,n+1)$.

By this method, we obtain $2n\ell$ $1$-factors and $(n+1)(\frac{m-2}{2}-\ell)$ $n$-star factors. By  Lemma~\ref{edgeF}, $I_{(n+1)}$ has a $(K_2,K_{1,n})-AURD(I_{(n+1)};n,0)$, so this produces $n$ more $1$-factors. Then, by Lemma~\ref{evenfilling}, we obtain $(m+n-1)$ more $1$-factors, which completes the decomposition.

\end{proof}

By Lemma~\ref{oddResult}, let $n\geq3$ be an odd integer and $m\geq3$ be an odd integer. We provided a solution to the existence of a $(K_2,K_{1,n})-URD(K_v;r,s)$ when $r \geq m+n-1$. By Lemma~\ref{evenResult}, if $n\geq3$ is an odd integer and $m\geq4$ is an even integer, we provided a solution to the existence of a $(K_2,K_{1,n})-URD(K_v;r,s)$ when $r \geq m+2n-1$. 

Since our construction requires a form of a cycle with $m$ vertices, $m$ must not be less than $3$. If $m<3$, then a $C_m$ does not exist, and we cannot use our construction. Second, if $m$ is odd, our construction requires a minimum of $m+n-1$ $1$-factors. Similarly, if $m$ is even, our construction requires a minimum of $m+2n-1$ $1$-factors. Thus, the following cases are excluded from our results.

\begin{itemize}
    \item $v=(n+1)$ and $v=2(n+1)$.
    \item Any $(r,s)$ pair with $r < m+n-1$ for odd $m \geq3$.
    \item Any $(r,s)$ pair with $r < m+2n-1$ for even $m \geq 4$.
\end{itemize}


\clearpage
\begin{thebibliography}{9}


\bibitem{AB}
R. J. R. Abel, {\em Some new near resolvable BIBDs with k = 7 and resolvable BIBDs with k = 5}, Australas.
J. Combin. 37,(2007), 141--146.


\bibitem{AGGZ}
R. J. R. Abel, G. Ge, M. Greig, and L. Zhu, {\em Resolvable BIBDs with a block size of $5$}, J. Stat. Plann.
Infer. \textbf{95} (2001), 49--65.

\bibitem{AG}
 R. J. R. Abel and M. Greig, {\em Some new $(v, 5, 1)$ RBIBDs and PBDs with block sizes $\equiv 1 \pmod{5}$},
Australas. J. Combin. \textbf{15} (1997), 177--202.


\bibitem{A} B. Alspach, {\em The wonderful Walecki construction}, Bull. Inst. Combin. Appl. {\bf 52} (2008), 7--20.


\bibitem{CC}
F. Chen and H. Cao, {\em Uniformly resolvable decompositions of $K_v$ into $K_2$ and $K_{1,3}$ graphs}, Discrete Math. \textbf{339} (2016), 2056--2062.


\bibitem{CD} C. J. Colbourn and J. H. Dinitz (eds.), {\em  Handbook of Combinatorial Designs}, Second Edition, Chapman and Hall/CRC, Boca Raton, FL, 2007.


\bibitem{DLD}
 J. H. Dinitz, A.C.H. Ling and P. Danziger, {\em Maximum Uniformly resolvable designs with  block sizes $2$ and $4$},
Discrete Math. {\bf 309} (2009), 4716--4721.


\bibitem{FMY}
S. C. Furino, Y. Miao, and J. X. Yin, {\em Frames and Resolvable Designs}, CRC Press, Boca Raton FL,
1996.



\bibitem{GM} M. Gionfriddo and S. Milici, {\em On the existence of uniformly resolvable decompositions of\/ $K_v$ and\/ $K_v-I$ into paths and kites},
Discrete Math. {\bf 313} (2013), 2830--2834.


\bibitem{GM1} M. Gionfriddo and S. Milici, {\em Uniformly resolvable $\cH$-designs with $\cH$=$\{P_3, P_4\}$}, Australas. J. Combin. {\bf 60} (2014), 325--332.

\bibitem{GM2} M. Gionfriddo and  S. Milici, {\em Uniformly resolvable $\{K_2, P_k\}$-designs with $k$=$\{3,4\}$},  Contrib. Discret. Math.  {\bf 10} (2015), 126--133.


\bibitem{MD} M. Keranen, D. Kreher, S. Milici, and A. Tripodi,
{\em Uniformly Resolvable Decompositions of $K_v$ into $1$-factors and $4$-stars} (2020) {\em ASTURALASIAN JOURNAL OF COMBINATORICS, 76(1), 55-72}.



\bibitem{KLMT} S. Kucukcifci, G. Lo Faro, S. Milici, and A. Tripodi,
{\em Resolvable\/ $3$-star designs}, Discrete Math. {\bf 338} (2015), 608--614.


\bibitem{KMT} S. Kucukcifci, S. Milici and Zs. Tuza,
{\em Maximum uniformly resolvable decompositions of\/ $K_v$ into\/ $3$-stars and\/ $3$-cycles},
 Discrete Math., in press doi:10.1016/j.disc.2014.05.016 .


\bibitem{LA} R. Laskar  and  B. Auerbach,
{\em On decomposition of r-partite graphs into edge-disjoint Hamilton circuits}, Discrete Math. {\bf 14} (1976), 265--268.


\bibitem{JAY} J. Lee and M. Keranen,
{\em Uniformly Resolvable Decompositions of $K_v-I$  into $5$-stars} (2023) Article is currently under review.

\bibitem{JAY2} J. Lee and M. Keranen,
{\em Uniformly Resolvable Decompositions of $K_vI$ into one $1$-factor and $n$-stars} (2025) Article is currently under review.



\bibitem{L} J. Liu,
{\em The equipartite Oberwolfach problem with uniform tables},
J. Comb. Theory  A {\bf 101} (2003), 20--34.


\bibitem{LMT} G. Lo Faro, S. Milici, and A. Tripodi,
{\em Uniformly resolvable decompositions of  into paths on two, three and four vertices}, Discrete Math. {\bf 338} (2015), 2212--2219.




\bibitem{Lu}
E. Lucas, {\em  R\'{e}cr\'{e}ations math\'{e}matiques}, Vol. $2$, Gauthier-Villars, Paris, 1883.



\bibitem{M} S. Milici, {\em A note on uniformly resolvable decompositions of\/ $K_v$ and\/ $K_v-I$ into\/ $2$-stars and\/ $4$-cycles},  Austalas. J. Combin. {\bf 56} (2013), 195--200.



\bibitem{MT} S. Milici and Zs. Tuza,
 {\em Uniformly resolvable decompositions of\/ $K_v$ into\/ $P_3$ and\/ $K_3$ graphs}, Discrete Math.
{\bf 331} (2014), 137--141.


\bibitem{R} R. Rees, {\em Uniformly resolvable pairwise balanced
designs with block sizes two and three}, J. Comb. Theory  A
{\bf 45} (1987), 207--225.


\bibitem{RS}  R. Rees and D. R. Stinson.
{\em On resolvable group divisible designs with block size $3$, Ars
Combinatoria} {\bf 23} (1987), 107--120.


\bibitem{SG}
 E. Schuster and G. Ge, {\em On uniformly resolvable designs with block sizes
$3$ and $4$ }, Design. Code. Cryptogr.  {\bf 57} (2010), 57--69.

\bibitem{WG}
 H. Wei and G. Ge, {\em Uniformly resolvable designs with block sizes
$3$ and $4$}, Discrete Math.  {\bf 339} (2016), 1069--1085.

\bibitem{Y}
M. L. Yu, {\em On tree factorizations of $K_n$},  J. Graph Theory   {\bf 17} (1993),  713--725.


\bibitem{ZDZ}
L. Zhu, B. Du, and X. B. Zhang, {\em A few more RBIBDs with $k = 5$ and $\lambda= 1$ }, Discrete Math. \textbf{97} (1991),
409--417.



\end{thebibliography}
\end{document}